\documentclass[11pt,reqno]{amsart}

\usepackage{geometry}
\usepackage[T1]{fontenc}
\usepackage[utf8]{inputenc}
\usepackage[english]{babel}
\usepackage{amsmath,amsthm,amssymb,mathrsfs}
\usepackage{cancel}
\numberwithin{equation}{section}
\usepackage{booktabs}
\usepackage{subcaption}
\usepackage[all]{xy}
\usepackage{tikz-cd, multirow}
\usepackage{enumitem}
\usepackage{array}
\usepackage{nicematrix}
\usepackage{graphicx}
\usepackage{makecell}
\usepackage{float}
\usepackage{mathtools}
\usepackage[english]{varioref}
\usepackage{appendix}
\usepackage{multirow}
\usepackage{xcolor}
\usepackage{enumitem}
\usepackage{fancyhdr}
\usepackage{hyperref}
\usepackage{longtable}
\DeclareFontFamily{U}{BOONDOX-calo}{\skewchar\font=45 }
\DeclareFontShape{U}{BOONDOX-calo}{m}{n}{
  <-> s*[1.05] BOONDOX-r-calo}{}
\DeclareFontShape{U}{BOONDOX-calo}{b}{n}{
  <-> s*[1.05] BOONDOX-b-calo}{}
\DeclareMathAlphabet{\mathcalboondox}{U}{BOONDOX-calo}{m}{n}
\SetMathAlphabet{\mathcalboondox}{bold}{U}{BOONDOX-calo}{b}{n}
\DeclareMathAlphabet{\mathbcalboondox}{U}{BOONDOX-calo}{b}{n}
\usepackage{chngcntr}
\usepackage[absolute,overlay]{textpos}
\usepackage{lipsum}
\usepackage{setspace}
\usepackage[justification=centering]{caption}
\counterwithin{figure}{section}
\hypersetup{ bookmarks=true, colorlinks=true, linkcolor=red, citecolor=blue, urlcolor=black}

\newtheorem{theorem}{Theorem}[section]
\newtheorem{lemma}[theorem]{Lemma}
\newtheorem{corollary}[theorem]{Corollary}
\newtheorem{proposition}[theorem]{Proposition}

\newcounter{maintheorem}                       

\newtheorem{mainth}[maintheorem]{Theorem}
\theoremstyle{remark}
\newtheorem{remark}[theorem]{Remark}

\theoremstyle{definition}
\newtheorem{definition}[theorem]{Definition}
\newtheorem{conjecture}[theorem]{Conjecture}

\newcounter{smallromans}

\newenvironment{romanenumerate}
{\begin{list}{{\normalfont\textrm{(\roman{smallromans})}}}%
  {\usecounter{smallromans}\setlength{\itemindent}{0cm}%
   \setlength{\leftmargin}{5.5ex}\setlength{\labelwidth}{5.5ex}%
   \setlength{\topsep}{.5ex}\setlength{\partopsep}{.5ex}%
   \setlength{\itemsep}{0.1ex}}}%
{\end{list}}

\newcommand{\rsimeq}{\mathbin{\rotatebox[origin=c]{90}{$\simeq$}}}
\newcommand{\lsimeq}{\mathbin{\rotatebox[origin=c]{-90}{$\simeq$}}}

\newcommand{\cA}{{\mathcal A}}

\newcommand{\cH}{{\mathcal H}}

\newcommand{\cO}{{\mathcal O}}

\newcommand{\cP}{{\mathcal P}}
\newcommand{\cG}{{\mathcal G}}

\newcommand{\bC}{{\mathbb C}}
\newcommand{\bR}{{\mathbb R}}
\newcommand{\bZ}{{\mathbb Z}}
\newcommand{\bQ}{{\mathbb Q}}
\newcommand{\bN}{{\mathbb N}}

\newcommand{\bH}{{\mathbb H}}
\newcommand{\bF}{{\mathbb F}}

\newcommand{\ol}[1]{\overline{#1}}

\allowdisplaybreaks

\makeatletter
\def\keywords{\xdef\@thefnmark{}\@footnotetext}
\makeatother

\begin{document}
\onehalfspacing

\title[The Kawamata-Morrison Cone Conjecture for GHV]{The Kawamata-Morrison Cone Conjecture for \\Generalized Hyperelliptic Variety}
\author{Martina Monti and Ana Quedo}
\address{Dipartimento di Matematica Federigo Enriques - UNIMI, Via Cesare Saldini, 50, 20133 Milano MI, Italia and  Laboratoire de Mathématiques et Applications - Université de Poitiers, UMR 7348 du CNRS, 11 bd Marie et Pierre Curie, 86073 Poitiers Cedex 9, Frances}
\email{martina.monti1@unimi.it}
\address{Instituto de Matemática Pura e Aplicada - IMPA, Estrada Dona Castorina, 110, Rio de Janeiro, RJ 22460-320, Brazil}
\email{ana.quedo@impa.br}
\maketitle

\begin{abstract}
    A Generalized Hyperelliptic Variety (GHV) is the quotient of an abelian variety by a free action of a finite group which does not contain any translation. These varieties are natural generalizations of bi-elliptic surfaces. In this paper we prove the Kawamata-Morrison Cone Conjecture for these manifolds using the analogous results established by Prendergast-Smith for abelian varieties.
\end{abstract}

\section{Introduction}
\subsection{The cone conjecture}
The investigation of the Ample Cone and the Cone of curves within algebraic projective varieties is pivotal in the realm of algebraic geometry, in particular within birational geometry. One of the core insights of to the Minimal Model Program is the realization that a wealth of information concerning morphisms from projective varieties to projective spaces is encapsulated in the nef cone. This cone, arising as the closure of the Ample Cone and dual to the Cone of curves, serves as a key focal point.
Given its convex nature, employing tools from Hyperbolic Geometry and Convex Geometry offers avenues for deeper comprehension of it. For instance, Fano varieties exhibit a simplicity in their nef cones since these cones are rational polyhedral cones \cite[Theorem 3.7]{KM98}, \emph{i.e.} their elements are $\mathbb{R}_{\geq 0}$-linear combinations of a finite set of integer vectors.
However, once left the world of Fano varieties, the cones may be very complicated. Of particular interest are varieties with numerically trivial canonical bundles ($K$-trivial varieties), where the nef cone can exhibit different behaviors, ranging from rational polyhedral structures to configurations with infinitely many extremal rays.
Nevertheless, there exists a pursuit for finiteness structure concerning the nef cone.
The Cone Conjecture, proposed first by Morrison \cite{Mor93} and then reformulated by Kawamata \cite{Kaw97}, deli\-neates the precise finiteness behavior of certain cones associated with $K$-trivial varieties. 
The conjecture predicts that for a $K$-trivial projective smooth variety $Y$ such cones can be covered by translates of a rational polyhedral cone by the action of automorphisms of $Y$ via their pull-back.  More precisely, let $N^1(Y)$ be the real vector space spanned by Cartier divisors modulo numerical equivalence. The closed cones $\text{Nef}(Y)$ and $\ol{\text{Mov}(Y)}$ in $N^1(Y)$ are the closures of the cones spanned by the ample and movable divisors, respectively. We denote by $\text{Nef}(Y)^e$ and $\ol{\text{Mov}(Y)}^e$ the intersection $\text{Nef}(Y)\cap \text{Eff}(Y)$ and $\ol{\text{Mov}(Y)}\cap \text{Eff}(Y)$, respectively, where $\text{Eff}(Y)$ is the cone of effective divisors.

\begin{conjecture} [Kawamata-Morrison] \label{conj}
Let $Y$ be a smooth projective variety with $K_Y \equiv 0$. Then:
\begin{romanenumerate}
 \item There exists a rational polyhedral cone $\Pi$ which is a fundamental domain for the action of the automorphism group \textup{Aut}$(Y)$ on the effective nef cone $\textup{Nef}(Y)^e$ in the following sense:
 \begin{itemize}
     \item [a.] $\textup{Nef}(Y)^e=\textup{Aut}(Y)\cdot \Pi$, \emph{i.e.} $\textup{Nef}(Y)^e=\bigcup\limits_{\varphi \in \textup{Aut}(Y) } \varphi^\ast \Pi$,
     \item[b.]  It holds $(\textup{Int}\Pi)\cap \varphi^\ast (\textup{Int}\Pi)=\emptyset$ unless
     $\varphi^\ast=id$ in \textup{GL}$(N^1(X))$.
 \end{itemize}

\item  There exists a rational polyhedral cone $\Pi'$ which is a fundamental domain (in the sense above) for the action of the birational automorphism group \textup{Bir}$(Y)$ on the effective movable cone $\ol{\text{Mov}(Y)}^e$ .   
\end{romanenumerate}
\end{conjecture}

The connections with birational geometry become more evident in this statement. The first item of the conjecture would imply that faces of $\text{Nef}(Y)^e$ corresponding to birational contractions or fiber space structures are finite 
up to automorphisms (see \cite[Section 1]{Tot10}), while the second one would imply, modulo standard conjectures of the Minimal Model Program, the finiteness of minimal models, up to isomorphisms (see \cite[Theorem 2.14]{CL14}).\\
Conjecture \ref{conj} is known as Kawamata's Cone Conjecture \cite{Kaw97}. While the original version, Morrison's version, of the Cone Conjecture \cite{Mor93} is stated for the cone $\text{Nef}(Y)^+$ and $\ol{\text{Mov}(Y)}^+$ which are the convex hull of $\text{Nef}(Y)\cap N^1(Y)_{\bQ}$ and $\ol{\text{Mov}(Y)}\cap N^1(Y)_{\bQ}$, respectively. For the nef cone, the connection between these two different versions of the Cone Conjecture can be explained by the following inclusion $\text{Nef}(Y)^e\subseteq \text{Nef}(Y)^+$ which is know in general for projective smooth $K$-trivial variety $Y$ (see for example \cite[Theorem 2.15]{LOP20}), while the reverse inclusion is still wide open. Specifically, the exact relation is unveiled by a seminal work of Looijenga from Convex Geometry, see Lemma \ref{thm1.1a}, which implies that part (i) of Conjecture \ref{conj} is is equivalent to the following one (see also \cite[Corollary 2.6]{GLW22}).
\begin{conjecture} \label{conj2}
Let $Y$ be a smooth projective variety with $K_Y \equiv 0$. 
\begin{enumerate}
\item There exists a rational polyhedral cone $\Pi$ which is a fundamental domain for the action of the automorphism group \textup{Aut}$(Y)$ on $\textup{Nef}(Y)^+$ (Morrison's version) .
\item $\textup{Nef}(Y)^+=\textup{Nef}(Y)^e$.
\end{enumerate}
\end{conjecture}
The idea behind splitting the Cone Conjecture for the nef cone into these parts lies in the fact that, in the setting of Convex Geometry, the convex hull of rational points emerges as a more natural object than the effective cones. A strategy to tackle this conjecture entails deploying general tools of Convex Geometry to establish part (1) of Conjecture \ref{conj2} (Morrison's version), followed by a deeper investigation into the geometric properties of the varieties to verify part (2) of Conjecture \ref{conj2}.
It is worth noting that a more generalized version of this conjecture exists for klt Calabi-Yau pairs, see \cite{Tot10}.\\
Over the years, the conjecture has spurred a flurry of research activity, leading to signi\-ficant advancements and conjectural extensions.
Conjecture \ref{conj} has been validated in numerous instances. In the realm of dimension $2$, pioneering proofs were initially provided by Sterk \cite{St85}, Namikawa \cite{Nam85}, Kawamata \cite{Kaw97}, and Totaro \cite{Tot10}. In higher dimensions, seve\-ral results have been established for distinct classes of varieties. Notably, Prendergast-Smith \cite{Pr12} verified the conjecture for abelian varieties, while Amerik-Verbistky \cite{Mark11}, \cite{AV17}, \cite{AV20} extended it to IHS manifolds. Moreover, albeit in very specific cases, the conjecture has been demonstrated for Calabi-Yau manifolds (see \cite{LP13}, \cite{Og14}, \cite{OY15}, and \cite{LOP17}).

\subsection{The cone conjecture under étale quotients}
The Beauville-Bogomolov decomposition theorem \cite{Beau83a} states that all complex projective manifolds $Y$ with numerically trivial canonical bundle admit a finite étale cover $X\rightarrow Y$ isomorphic to the product of abelian varieties, simply connected Calabi-Yau manifolds and Irreducible Holomorphic Symplectic Manifolds (IHSM). This prompts a natural question: if Conjecture \ref{conj} holds true for $X$, can we infer its validity for $Y$? Pacienza and Sarti in \cite{PS23} provided an affirmative response for coverings of prime degree of IHS type, in these cases the resulting manifolds are called Enriques manifolds. In this article, we explore the scenario where the cover $X$ is an abelian variety. More precisely, a \emph{Generalized Hyperelliptic Variety (GHV)} is the quotient $Y=X/G$ where $X$ is an abelian variety and $G\le \text{Aut}(Y)$ is a finite group acting freely on $X$ without any translations, (see Remark \ref{GHV2}). By construction, GHVs have an infinite fundamental group, and their canonical bundle is numerically trivial. These varieties, first introduced by Lange \cite{La}, are natural generalizations of bi-elliptic surfaces. Subsequently, various mathematicians have continued to investigate these mani\-folds, as evidenced by works \cite{CC}, \cite{CD2}, \cite{De}, and \cite{M23}. 
It is important to note that there is a full classification of these varieties up to dimension $4$: for surfaces the seminal works of Ba\-gnera and de Franchis \cite{BdF}, Enriques and Severi \cite{ES09}, \cite{ES10};  for threefolds Lange \cite{La} and Catanese and Demleitner \cite{CD2}; for fourfolds Demleitner \cite{De}. In particular, their existence is guaranteed in all dimensional cases, see \cite{La} and \cite{Vi21}.\\

\noindent The principal goal of this article is to establish Conjecture \ref{conj2} for the wide class of GHVs. Specifically, our main result is the following.

\begin{mainth}\label{main1}
Let $Y=X/G$ be a Generalized Hyperelliptic Variety. Then, part (i) of Conjecture \ref{conj} is verified and $\ol{\textup{Mov}(Y)}^e=\textup{Nef}(Y)^e=\textup{Nef}(Y)^+=\textup{Nef}(Y)$. In particular, also part (ii) of Conjecture \ref{conj} is verified.
\end{mainth}

A core idea underlying the proof of Theorem \ref{main1} lies in the ability to describe the nef cone of $Y$ in terms of the $G$-invariant nef cone of $X$ and to establish connections between the automorphisms of $Y$ and the one in the normalizer $\text{N}_{\text{Aut}(X)}(G)$. Specifically, assuming the conjecture holds true for a variety $X$, one may seek to establish the existence of a rational polyhedral fundamental domain $\Pi$ for the action of $\text{N}_{\text{Aut}(X)}(G)$ (or a subgroup $H$ thereof) on $\bigl(\text{Nef}(X)^G\bigr)^+$ and subsequently project $\Pi$ under $\pi_\ast$, yielding a rational polyhedral fundamental domain for the action of $\text{Aut}(X)$ on $\text{Nef}(Y)^+$. Actually, we demonstrate that it suffices to provide a rational polyhedral cone $\Pi\subset \bigl(\text{Nef}(X)^G\bigr)^+$ such that $\text{Amp}(X)^G\subset H\cdot \Pi$ for some $H\le\text{N}_{\text{Aut}(X)}(G)$ to obtain a rational polyhedral fundamental domain for the nef cone of the quotient $Y=X/G$, see Proposition \ref{conjquotient}.

For abelian varieties $X$, Prendergast-Smith \cite{Pr12} translates the conjecture \ref{conj} for the nef cone into a well-known problem of Convex Geometry concerning the existence of a rational polyhedral fundamental domain for the action of arithmetic groups on homogeneous self-dual cones (we refer to Section \ref{sec.convexgeo} and Theorem \ref{thm1.1} for further details). 
Our approach mirrors this connection with Convex Geometry. Specifically, we establish that the $G$-invariant $\bR$-algebra $\text{End}_{\bR}(X)^G$ assumes a certain matrix algebraic structure (Theo\-rem \ref{dec-Ginvarinatalgebra}). Moreover, by using that $\textup{C}_{\textup{Aut}(X)}(G)$ and $\text{Amp}(X)^G$ can be embedded in this algebra, we ascertain their properties to be an arithmetic group and a homogeneous self-dual cone, respectively. This connection with the theory of reduction enables us to invoke the main result of this theory (Theorem \ref{thm1.1}) which guarantees the existence of a rational polyhedral cone $\Pi\subset \bigl(\text{Nef}(X)^G\bigr)^+$ such that $\text{Amp}(X)^G\subset \textup{C}_{\textup{Aut}(X)}(G)\cdot \Pi$. Additionally, for GHVs $Y$ (as for abelian varieties) we have $\ol{\text{Mov}(Y)}^e=\textup{Nef}(Y)^e=\textup{Nef}(Y)^+=\text{Nef}(Y)$ (Proposition \ref{conjGHV}), thereby the validity of the Cone Conjecture for the nef cone implies the one for the movable cone.\\

This paper is structured as follows. In Section \ref{sec.GHV}, we lay the groundwork by introducing GHVs and revisiting their key properties. Sections \ref{sec.abelian} and \ref{sec.convexgeo} serve to recapitulate the essential findings in abelian varieties and Convex Geometry, including the Reduction Theory, which are crucial to our proof. Section \ref{sec.conjGHV} is dedicated to reformulating the cone conjecture for étale quotients, with a particular focus on GHVs. The heart of our article lies in Section \ref{sec.proof}, wherein we present the proof of our main result, Theorem \ref{main1}, as explained previously.\\

\noindent \textbf{Acknowledgements.}
The authors gratefully acknowledge Alessandra Sarti for introducing the problem and engaging in discussions throughout the writing process. Special thanks are extended to Gianluca Pacienza for suggesting the topic. Additionally, the authors express their gratitude to Alice Garbagnati and Carolina Araujo for their invalua\-ble feedback on the final version of the article. \\
The first author acknowledges Cécilie Gachet for insightful discussions and contributions during a winter school in Rennes and beyond. Gratitude is also extended to Enrica Floris for her helpful insights. She would like to thank the Departments of Mathematics of  Université de Poitiers and Universidad de la Frontera for the financial support during her stay in Chile, which facilitated productive discussions with esteemed mathematicians for the present paper. Acknowledgements are extended to the Departments of Mathematics of Università degli Studi di Milano for the fundings. The first author is supported by the “VINCI 2022-C2-62” grant issued by the Université Franco-Italianne. The second author expresses gratitude to Mikhail Belolipetsky for elucidating various perspectives on arithmetic groups, a crucial concept in this research. Furthermore, gratitude is expressed to CAPES (Coordenação de Aperfeiçoamento de Pessoal de Nível Superior) for the financial support received during her stay in Poitiers, which facilitated the collaboration that resulted in the successful completion of this paper.


\section{Generalized Hyperelliptic Varieties}\label{sec.GHV}
\noindent In this section we introduce our main objects and we recall the main proprieties.

\begin{definition}
A \emph{complex torus} is the quotient of $\bC^n$ by the action of translations of a lattice $\Lambda\simeq \bZ^{2n}$.
If the torus has a positive line bundle it is called \emph{abelian variety}.
\end{definition}

\begin{definition}
A \emph{Generalized Hyperelliptic Variety (GHV)} is a manifold not isomorphic to an abelian variety but which admits an abelian variety as étale Galois cover.
\end{definition}

\begin{remark}\label{defGHV}
    If $Y$ is a GHV, by definition, there exists an abelian variety $X$ and a group $G\le \text{Aut}(X)$ acting freely on $X$ such that $Y=X/G$. In particular, since $Y$ is not isomorphic to an abelian variety, $G$ does not contain only translations. In fact, one can assume that $G$ does not contain any translation. Indeed, if  $ \langle id \rangle \neq G_0 \unlhd G$ is the subgroup of translations then $Y$ is also the quotient of the abelian variety $X'=X/G_0$ by the free finite action of the group $G'=G/G_0$ which does not contain any translation.
\end{remark}


\begin{remark}\label{GHV2}
We observe that every projective compact manifold $Y$ with numerically trivial canonical bundle whose universal cover is $\bC^n$ is a GHV. Indeed, the fundamental group $\pi_1(Y)$ is isomorphic to a torsion-free crystallographic group $\Gamma \in \text{Iso}(\bC^n)$, \emph{i.e.} it is discrete under compact-open topology, $\bC^n/\Gamma$ is compact and $\Gamma$ acts freely. By the first Bierbebach theorem (\cite[Theorem 2.1]{Sz}), we know that $\Gamma$ is characterized by the following exact sequence: $$1\longrightarrow \Lambda \longrightarrow \Gamma \longrightarrow G \longrightarrow 1$$
where $\Lambda$ is a subgroup of translations (which is maximal abelian and normal of finite index) and $G$ is a finite group. Therefore, $Y=\bC^n/\Gamma$ can be also obtained as the quotient of the complex torus $T=\bC^n/\Lambda$ by the action of the finite group $G=\Gamma/\Lambda$ which acts freely and does not contain any translation (since $\Gamma$ acts freely and $\Lambda$ is the maximal abelian and normal of finite index). Moreover, since $Y$ is projective and $\pi: T\longrightarrow Y$ is finite, we get that $T$ is a projective complex torus and so $Y$ is a GHV.
\end{remark}

\begin{proposition}\label{p1.9}
    Let $Y=X/G$ be a GHV and $D$ be a Cartier divisor on it. We denote by $[D]_{num}$ the class of $D$ in $N^1(Y)$. Then $[D]_{num}$ in $\textup{Eff}(Y)$ if and only if $[D]_{num}$ in ${\textup{Nef}(Y)}$.
\end{proposition}
\begin{proof}
We recall that the statement is true for abelian varieties, \cite[Lemma 1.1]{Bau98}. Let $\pi: X \longrightarrow Y$ be the étale Galois cover.
The statement easily follows using that $D$ is nef/effective on $Y$ if and only if $\pi^\ast D$ is nef/effective on $X$.
\end{proof}

\begin{proposition}\label{p1.7}
Let $Y=X/G$ be a GHV manifold. Then:
\begin{romanenumerate}
\item It holds $\textup{Aut}(Y)\simeq \dfrac{\textup{N}_{\textup{Aut}(X)}(G)}{G}$
where $\textup{N}_{\textup{Aut}(X)}(G)$ is the normalizer of $G$ in $\textup{Aut}(X)$.
\item The pull back map $\pi^\ast$ defines the following isomorphisms: $${\textup{Amp}(Y)}\simeq {\textup{Amp}(X)}^G \qquad \textup{Nef}(Y)\simeq \textup{Nef}(X)^G.$$
\end{romanenumerate}

\end{proposition}
\begin{proof}
\begin{romanenumerate}
    \item The following homomorphism of groups is well-defined
    \begin{center}
        \begin{tikzcd}
            \Phi\colon N_{\text{Aut}(X)}(G) \arrow[r] & \text{Aut}(Y) \\
            \alpha_X \arrow[r,mapsto] & \alpha_Y
        \end{tikzcd}
    \end{center}
    where $\alpha_Y$ is such that \begin{tikzcd}
        X \arrow[r,"\alpha_X"] \arrow[d] & X \arrow[d] \\
        Y \arrow[r,"\alpha_Y"'] & Y
    \end{tikzcd} is commutative.
    By \cite[Corollary 1.7]{M23}, it follows that $\Phi$ is surjective and Ker$(\Phi)=G$. Hence, we conclude by the first theorem of isomorphism.

    \item  Let us consider the following homomorphism 
    $$ \begin{tikzcd}
           \pi^\ast \colon N^1(Y) \arrow[r] & N^1(X)^G \\
           D \arrow[r, mapsto] & \pi^\ast D
    \end{tikzcd}$$
which is an isomorphism of vector spaces.
Since $\pi$ is a finite morphism, by \cite[Proposition 4.4]{Hart2} the pull back $\pi^\ast D$ with $D\in N^1(Y)$ is ample if and only if $D$ is ample. Thus we have the first isomorphism. The second one follows since the nef cone is the closure of the ample cone.
\end{romanenumerate}
\end{proof}

\section{Preliminaries on Abelian varieties}\label{sec.abelian}
 \noindent In this section we recall the main proprieties of abelian varieties that will be useful in the subsequent sections, we refer to \cite{BL}. 



\subsection{Endomorphism algebra of abelian varieties}

\begin{definition}
An \emph{automorphism} of an abelian variety $X$ is a biholomorphic map.
An \emph{endomorphism} of an abelian variety $X$ is an automorphism that is compatible with the group structure.   
\end{definition}

We denote the group of automorphism of $X$ by $\textup{Aut}(X)$. We recall that any $\varphi$ in $\text{Aut}(X)$ is a composition of an endomorphism and a translation.\\
Let us denote by End$(X)$ the ring of endomorphisms of $X$ and by End$_{\bQ}(X)=\text{End}(X)\otimes\bQ$ its extension on $\bQ$, called \emph{$\bQ$-endomorphism algebra}. In this section, we recall some properties of the $\bQ$-endomorphism algebra, for further details see \cite[Chapter 5]{BL}. 



\begin{definition}
An isogeny $\varphi: X \rightarrow Y$ of abelian varieties is a surjective endomorphism with finite kernel. If such $\varphi$ exists then $X$ and $Y$ are said to be \emph{isogenous}.
\end{definition}

\begin{remark}
An isogeny $\varphi\in Hom(X,Y)$ is invertible only in $\text{Hom}_{\bQ}(X,Y)$, see \cite[Proposition 1.2.6]{BL}.
\end{remark}

\begin{definition}[Section 2.4, \cite{BL}]
The \emph{dual abelian variety} $\hat{X}$ of $X=\bC^n/\Lambda$ is the quotient of the $\bC$-antilinear form $\ol{\Omega}:=\textup{Hom}_{\ol{\bC}}(\bC^n,\bC)$ by the action of the dual lattice $\hat{\Lambda}=\{l\in \ol{\Omega} \colon \text{Im}l(\lambda) \in \bZ \text{ for every $\lambda \in \Lambda$}\}$.   
\end{definition}

 By \cite[Proposition 2.4.1]{BL}:  $\hat{X}\simeq \text{Pic}^0(X)=\text{ker}(c_1)$ where $c_1: \text{Pic}(X)\longrightarrow NS(X)$ is the first Chern class.
Given an homomorphism $\varphi\colon X\rightarrow X$ with analytic re\-presentation $\widetilde{\varphi}:\bC^n\longrightarrow \bC^n$, the (anti)-dual map $\widetilde{\varphi}^\vee:\ol{\Omega}\longrightarrow \ol{\Omega}$ induces a homomorphism $\hat{\varphi}: \hat{X} \longrightarrow \hat{X}$ called the \emph{dual map} of $\varphi$, see \cite[Section 2.4]{BL}. 

Let $D\in \textup{Pic}(X)$, for any point $x\in X$ the line bundle $t^\ast_x D \otimes D^{-1}$ has zero first Chern class, where $t_x$ is the translation by $x$. We get a group homomorphism, as follows, for any $D\in \text{Pic}(X)$:

\begin{equation}\label{mapondual}
\begin{tikzcd}
       \phi_{D} \colon X \arrow[r] & \hat{X}\simeq\text{Pic}^0(X) \\
       x \arrow[r,mapsto] & t^\ast_x D \otimes D^{-1}
\end{tikzcd}
\end{equation}

Let us denote by $L$ an ample line bundle on $X$. Since $L$ is ample, $\phi_{L} \colon X\longrightarrow \hat{X}$ is an isogeny and so the inverse $\phi_L^{-1} \colon \hat{X}\longrightarrow X$ is well-defined in $\text{Hom}_{\bQ}(\hat{X},X)$. This allows to define an involution on End$_{\bQ}(X)$ as follows:
\begin{equation}\label{rosati}
\begin{tikzcd}
' \colon \textup{End}_{\bQ}(X) \arrow[r] & \textup{End}_{\bQ}(X) \\
 \varphi\arrow[r, mapsto] & \varphi':=\phi_L^{-1} \hat{\varphi} \phi_L .
\end{tikzcd} \end{equation}

\noindent The involution above is called \emph{Rosati involution}. It is positive-definite with respect to the reduced trace Tr$_{\bQ}$ over $\bQ$, \emph{i.e.} $\forall \varphi\in \text{End}_{\bQ}(X)$ then Tr$_{\bQ}(\varphi\circ \varphi')>0$, see \cite[Theorem 5.1.8]{BL}. 

The following result, well-known to people working with abelian varieties, describes the $\bR$-algebra $\text{End}_{\bR}(X):=\text{End}_{\bQ}(X)\otimes \bR$ as a product of certain matrices algebras. The result follows combing the Poincarè reducibility Theorem \cite[Theorem 5.3.7]{BL} and the classification of $\bQ$-division algebras with a positive-definite involution due to Albert \cite[IV.21 Theorem 2, page 201]{Mum}.

\begin{theorem}[Corollary 3.5 \cite{Pr12}]\label{dec-endo-algebra}
Let $X$ be an abelian variety. Then, we have the following isomorphism of $\bR$-algebra:
$$ \bigl(\textup{End}_{\bR}(X), '\bigr) \xrightarrow{\quad \simeq \quad} \bigl(\prod\limits_i \textup{Mat}_{r_i}(\bR) \times \prod\limits_j \textup{Mat}_{s_i}(\bC) \times \prod\limits_k \textup{Mat}_{t_k}(\bH),\dagger \bigr) $$ 
where the Rosati involution $'$ is sent to the positive-definite involution $\dagger$ given by the conjugate transpose on each factor.
\end{theorem} 

\subsection{Nef cone of an abelian variety}
In this section, we recall the description of the nef cone inside the $\bR$-algebra $\text{End}_{\bR}(X)$.

\begin{theorem} [Section 5.2 \cite{BL}] \label{thm2.3}
Let $X$ be an abelian variety and $L$ be the ample line bundle defining the Rosati involution $'$ in \eqref{rosati}. There is an embedding 
\begin{center} 
\begin{tikzcd}
f \colon N^1(X) \arrow[r, hookrightarrow] & \textup{End}_{\bR}(X) \\
 D \arrow[r, mapsto] & \phi_L^{-1} \phi_D 
\end{tikzcd}
\end{center}
where $\phi_L$ and $\phi_D$ are defined as in \eqref{mapondual}.
In particular: $$N^1(X)\simeq \{ \varphi\in  \textup{End}_{\bR}(X) \colon  \varphi=\varphi'\}:=\textup{End}^s_{\bR}(X). $$ 
\end{theorem}

Following \cite{Pr12} we recall that the action by pullback of $\text{Aut}(X)$ on $N^1(X)$ can be extended as an action of $\text{End}_{\bR}(X)^\times$ on $\text{End}_{\bR}(X)$.

\begin{lemma}[Corollary 2.4.6. (d) \cite{BL}] \label{pullbackrosati}
Let $X$ be an abelian variety. If $\varphi\in \textup{End}(X)$ and $D\in \textup{Pic}(X)$, then
\begin{equation}\label{dual2}
    \phi_{\varphi^\ast D}=\hat{\varphi} \circ \phi_D \circ \varphi.
\end{equation}
\end{lemma}

\begin{theorem}[Section $4$ \cite{Pr12}] \label{alpha}
Let $X$ be an abelian variety. Then, the group $\textup{End}_{\bR}(X)^\times$ acts on $\textup{End}_{\bR}(X)$ as follows:
\begin{equation}\label{eq2.2}
\begin{tikzcd}
     \alpha \colon (\textup{End}_{\bQ}(X)\otimes \bR)^\times \arrow[r] &  \textup{GL}(\textup{End}_{\bQ}(X)\otimes \bR) \\
     \varphi  \arrow[r,mapsto] &  \alpha(\varphi)\colon l  \mapsto \varphi'\circ l \circ \varphi
\end{tikzcd}    
\end{equation}
where $'$ is defined in \eqref{rosati} and this extends the action of $\textup{Aut}(X)$ on $N^1(X)$ by pullbacks.
\end{theorem}
\begin{proof}
It is easy to prove that $\alpha$ defines an action. We prove that it preserves $N^1(X)$ and it acts on it by pullbacks, \emph{i.e.} for every $\varphi\in \text{End}_{\bR}(X)^\times$ the following diagram is commutative:
\begin{center}
    \begin{tikzcd}
        N^1(X) \arrow[r, hookrightarrow, "f"] \arrow[d, "\varphi^\ast" ' ]& \text{End}_{\bR}(X) \arrow[d,"\alpha(\varphi)"] \\
        N^1(X) \arrow[r, hookrightarrow, "f" '] & \text{End}_{\bR}(X)
    \end{tikzcd}
\end{center}
where $f$ is defined in Theorem \ref{thm2.3}.
For every $D\in N^1(X)$ we have:
$$ \alpha(\varphi)(f(D))=\varphi' f(D) \varphi= \varphi' \phi^{-1}_L \phi_D \varphi \overset{\eqref{rosati}}{=} \phi^{-1}_L \hat{\varphi}\phi_{L}\phi^{-1}_L \phi_D \varphi \overset{\eqref{dual2}}{=}\phi^{-1}_L \phi_{\varphi^\ast D}=f(\varphi^\ast D). $$
\end{proof}

\section{Reduction theory in Convex Geometry}\label{sec.convexgeo}
\noindent In this section we recall the main result of reduction theory for arithmetic groups acting on homogeneous self-dual cones.\\

\noindent Let $V$ be a finite-dimensional $\bR$-vector space. 

\begin{definition}
We say that $V$ has a $k$-structure for a subfield $k\subset \bR$  if it is obtained by extension of scalars from a vector space $V_k$ over $k$.
\end{definition}

\noindent A $C \subset V \setminus\{0\}$ is a \emph{cone} if it is a non-degenerate convex cone. 

\begin{definition}\label{def.hull}
The \emph{convex hull} of $C$ is the cone of convex combination of points in $C$, \emph{i.e.} convhull$(C)=\{ \sum a_i c_i \quad | \quad a_i\in \bR_{\geq 0}, \sum\limits_i a_i=1 \text{ and } c_i\in C\}$. \\
The \emph{rational hull of $C$}, denoted by $C^+$, is the convex hull of the rational points in $\ol{C}$, \emph{i.e.} $C^+=\text{ convhull }(\ol{C}\cap V_{\bQ})$.   
\end{definition}

\noindent We denote the group of transformations of $C$ by $\text{Aut}(C)=\{ \varphi \in \text{GL}(V) \colon \varphi(C)=C\}$.

\begin{definition} A cone $C$ is said to be \emph{homogeneous} if $\text{Aut}(C)$ acts transitively on it, \emph{i.e.} for every $x,y \in C$ there exists $\varphi \in \text{Aut}(C)$ such that $\varphi(x)=y$.
\end{definition}

\noindent Let $\ol{C^\vee} \subset V^\vee$ be the set of linear forms in $V^\vee$ that are non-negative on $C$. The \emph{dual cone} $C^\vee$ is the interior of $\ol{C^\vee} \setminus\{0\}$. 

\begin{definition} 
A cone C is said to be \emph{self-dual} if there exists a positive-definite form on $V$ such that the resulting isomorphism between $V$ and $V^\vee$ transforms $C$ into $C^\vee$.
\end{definition}

\noindent Due to a result of Vinberg we known that homogeneous self-dual convex cones can be completely classified into a small number of cases, see \cite{Vin63}.

\begin{definition}
Let $C_i\subset V_i$ cone in the vector space $V_i$ for $i=1,2$. We define \emph{the direct sum $C_1$ and $C_2$} in vector spaces $V_1\oplus V_2$ to be the cone $C_1 \oplus C_2 :=\{v_1 +v_2 \in V_1\oplus V_2|v_i \in C_i\}$ and call a cone \emph{indecomposable} if it cannot be written as the direct sum of two nontrivial cones.
\end{definition}

\begin{theorem} [Remark 1.11 \cite{AMRT}] \label{thm1.2}
Any convex cone $C\subset V$ can be written as a direct sum $\oplus_i C_i$ of indecomposable cones. The product $\prod\limits_i\textup{Aut}(C_i)$ is a finite-index subgroup of $\textup{Aut}(C)$. The cones $C_i$ are homogeneous and self-dual if and only if $C$ is too. Any indecomposable homogeneous self-dual cone is isomorphic to one of the following:
\begin{enumerate}
\item[1.] the cone $\cP_r(\bR)$ of positive-definite matrices in the space $\cH_r(\bR)$ of $r\times r$ real symmetric matrices;
\item[2.] the cone $\cP_r(\bC)$ of positive-definite matrices in the space $\cH_r(\bC)$ of $r\times r$ complex symmetric matrices; 
\item[3.] the cone $\cP_r(\bH)$ of positive-definite matrices in the space $\cH_r(\bH)$ of $r\times r$ quaternionic symmetric matrices;
\item[4.] the spherical cone $\{(x_0, \dots, x_n) \in \bR^{n+1} \quad | \quad  x_0 > \sqrt{x^2_1 + \dots x^2_n}\}$ ;
\item[5.] the $27$-dimensional cone of positive-definite $3 \times 3$ octonionic Hermitian matrices.
\end{enumerate}
The inner product for which the cone is self-dual is $\langle x, y\rangle = Tr(xy^\ast)$ in all cases except $4$, and the usual inner product on $\bR^{n+1}$ in case $4$.
\end{theorem}

Moreover, Vinberg \cite{Vin65} computed the automorphism groups of all the cones in the list of Theorem \ref{thm1.2}. In particular, we have the following result.

\begin{theorem}[\cite{Vin65}] \label{groupofcone}
Let $C$ be one of the cones $\cP_r(\bF)$ in the previous theorem where $\bF = \bR, \bC, \text{ or } \bH$. The identity component $\textup{Aut}(C)^0$ of the automorphism group of $C$ consists of all $\bR$-linear transformations of $\cH_r(\bF)$ of the form $D \xmapsto{\qquad}  M^\dag DM$ for some $M \in GL(r,\bF)$ where $M^\dag$ is the conjugate transpose.
\end{theorem}

\begin{definition}
An \emph{algebraic group $\cG$ over a field $k$} is an algebraic variety over $k$ endowed with a group structure such that the following homomorphisms:
\begin{center}
    \begin{tikzcd}
        \mu\colon G\times G \arrow[r] & G & i\colon G\arrow[r] & G \\
        (g,h)\arrow[r,mapsto] & gh & g\arrow[r,mapsto] & g^{-1}
    \end{tikzcd}
\end{center}
are morphisms of varieties.\\
An algebraic group $\cG$ is said to be \emph{defined over a subfield $K\subset k$} if the polynomial equations defining it have coefficients in $K$. We denote the underlying structure of $K$-variety of $\cG$ by $\cG(K)$.\\
An algebraic group $\cG$ is said to be a \emph{linear algebraic group} if it admits a closed (with respect to the Zarisky topology) embedding $\rho: \cG \hookrightarrow \textup{GL}(n,k)$ for some $n\in \bN$, \emph{i.e.} $\rho(\cG):=G(k)$ is a subgroup defined by polynomial equations with coefficient in $k$. 
\end{definition}

\noindent Another basic theorem about the automorphism group of a homogeneous self-dual cone is due to Vinberg \cite{Vin65}.

\begin{theorem}[\cite{Vin65}] \label{reductive}
Let $C \subset V$ be a self-dual convex cone. Then the automorphism group $\textup{Aut}(C)$ is the group of real points of a reductive algebraic group.
\end{theorem}


\begin{definition}[Section 7.C, \cite{Bor19}]\label{defari}
Let $\cG$ be an algebraic linear group in $\text{GL}(n,\bC)$ for some $n$ defined over $\bQ$. We define $\cG({\bZ}):=\cG\cap \text{GL}(n,\bZ)$. A subgroup $\Gamma \subset \cG({\bQ})$ is said to be \emph{arithmetic} if it is commensurable with $\cG(\bZ)$, \emph{i.e.} $\cG(\bZ)\cap \Gamma$ is of finite index in both $\cG(\bZ)$ and $\Gamma$.
\end{definition}

\begin{remark}
It is proved in \cite[Section 7.C]{Bor19} that the property of being arithmetic it is invariant under $\bQ$-isomorphisms.
\end{remark}

\begin{definition}
Let $V$ be a finite-dimensional vector space over $\bR$ and $C$ be a cone. Let $\Gamma\le \textup{GL}(V)$ be a group such that preserves the cone $C$, then a \emph{fundamental domain} for the action of $\Gamma$ on $C$ is a subset $\Pi\subset C$ such that
\begin{enumerate}
    \item [a.] $\Gamma\cdot \Pi=C$, \emph{i.e} $\bigcup\limits_{\gamma\in \Gamma} \gamma(\Pi)=C$;
    \item [b.]  $\gamma(\Pi) \cap \Pi \neq \emptyset$ has non-empty interior if and only if $\gamma\in
ker(\Gamma\longrightarrow \textup{GL}(V))$.
\end{enumerate}
\end{definition}

The basic problem in reduction theory, which dates back to Minkowski, is the following: given a homogeneous self-dual cone $C$ and an arithmetic group $\Gamma \le \mathcalboondox{Aut}(C)$ there exists a fundamental domain for the action of $\Gamma$ on $C$?
Borel has produced a theory of coarse fundamental domains (called Siegel sets) for arithmetic subgroups $\Gamma$. This theory has provided a tool to Ash to show the existence of a fundamental domain for actions of arithmetic groups on homogeneous self-dual cones. 

\begin{definition}
    A cone $C\subseteq V$, with $\text{dim } V = n$, is said to be \emph{polyhedral} if it is finitely-generated, \emph{i.e.} there is a set of vectors $\{v_1,\dots,v_k\}\in V$ such that $C=\{a_{1}v_{1}+\cdots +a_{k}v_{k}\mid a_{i}\in \bR_{>0},v_{i}\in \bR ^{n}\}$.\\
    A polyhedral cone is said to be \emph{rational} when it is generated by integers vectors, \emph{i.e.} $C=\{a_{1}v_{1}+\cdots +a_{k}v_{k}\mid a_{i}\in \bR _{>0},v_{i}\in \mathbb {\bZ} ^{n}\}$
\end{definition}

\begin{theorem}[\cite{AMRT}] \label{thm1.1}
Let $C$ be a homogeneous self-dual cone in a real vector space $V$ with $\bQ$-structure. Let $\textup{Aut}(C)$ be the automorphism group of $C$ and $\mathcalboondox{Aut}(C)$ be the associated reductive algebraic group which exists in view of Theorem \ref{reductive}. Assume that the connected component of identity $\mathcalboondox{Aut}(C)^0$ is defined over $\bQ$. Then, for any arithmetic subgroup $\Gamma$ of $\mathcalboondox{Aut}(C)^0$ there exists a rational polyhedral cone $\Pi \subset C^+$ such that $(\Gamma \cdot \Pi) \cap C = C$.
\end{theorem}

\noindent As Ash pointed out in \cite[pag. 75]{AMRT} starting from $\Pi$ is possible to construct a rational fundamental domain for the action of $\Gamma$ on $C$. \\
We also recall the following lemma which is seminal work of Looijenga \cite[Definition-Proposition 4.1, Application 4.14]{Loo}.

\begin{lemma} \label{thm1.1a}
Let $\Lambda$ be a finitely generated free $\bZ$-module, and let $C$ be a strict open cone in the $\bR$-vector space $\Lambda_{\bR} := \Lambda \otimes \bR$. Let $C^+$ be the convex hull of $ \ol{C} \cap \Lambda_{\bQ}$. Let $(C^\vee)^\circ \subset (\Lambda_{\bR})^\vee$ be the interior of the dual cone of $C$. Let $\Gamma$ be a subgroup of $\textup{GL}(\Lambda)$ which preserves the cone $C$. Suppose that
\begin{itemize}
    \item there exists a rational polyhedral cone $\Pi \subset C^+$ such that $\Gamma\cdot \Pi\supset C$;
    \item  there exists an element $\eta \in (C^\vee)^\circ \cap (\Lambda_{\bQ})^\vee$ whose stabilizer in $\Gamma$ (with respect to the dual action of $\Gamma$ on $(\Lambda_{\bQ})^\vee$) is trivial.
\end{itemize}
Then $\Gamma\cdot \Pi = C^+$, and in fact there exists a rational polyhedral cone $\Pi' \subset C^+$ which is a fundamental domain for the action of $\Gamma$ on $C^+$.
\end{lemma}

\section{The cone conjecture for GHV}\label{sec.conjGHV}
In this section, we see how the Conjecture \ref{conj} can be reformulated for étale quotients and in particular for GHV.

\begin{lemma}\label{eta}
    Let $Y$ be a normal projective variety. Then there exists $\eta\in (\textup{Amp}(Y)^\vee)^\circ \cap (N^1(Y)_{\bQ})^\vee$ such that its stabilizer for the action of $\textup{Aut}(Y)$ (on $(N^1(Y)_{\bQ})^\vee$) is trivial.
\end{lemma}
\begin{proof}
    See the proof of \cite[Proposition 2.3]{GLW22}
\end{proof}

\begin{proposition}\label{conjquotient}
Let $X$ be a compact projective manifold and $G\le \textup{Aut}(X)$ a finite group that acts freely on it. We denote $\pi: X\longrightarrow Y=X/G$. Assume the existence of a rational polyhedral cone $\Pi\subset \textup{Nef}(X)^G\cap \textup{Eff}(X)$ such that $\textup{Amp}(X)^G\subset H\cdot \Pi$ for some $H\le \textup{N}_{\textup{Aut}(X)}(G)$. Then $Y$ satisfies conjecture \ref{conj2}, \emph{i.e.} the Morrison Cone Conjecture and $\textup{Nef}(Y)^e=\textup{Nef}(X)^+$.
\end{proposition}
\begin{proof}
Let us consider $\pi_\ast(\Pi)$: it defines a rational polyhedral cone in $\text{Nef}(Y)^e\subseteq \text{Nef}(Y)^+$ (the last inclusion holds by \cite[Lemma 5.1]{LW23}) such that: $$\text{Amp}(Y)\subset (H/G)\cdot \pi_\ast(\Pi)\subset \text{Aut}(Y)\cdot \pi_\ast(\Pi). $$ These inclusions together with Lemma \ref{eta} allow us to apply  Lemma \ref{thm1.1a} and obtain that $\text{Aut}(X)\cdot \pi_\ast(\Pi)=\text{Nef}(Y)^+$. In particular, there is a rational fundamental domain for the action of $\text{Aut}(Y)$ on $\text{Nef}(Y)^+$. Therefore, part (1) of conjecture \ref{conj2} is verified. 
Moreover, since $\pi_\ast(\Pi)\subset \text{Nef}(Y)^e$ we have $\text{Nef}(X)^+=\text{Aut}(X)\cdot \pi_\ast(\Pi)\subset \text{Nef}(Y)^e$, thus we obtain $\text{Nef}(X)^e=\text{Nef}(X)^+$. Therefore part (2) of conjecture \ref{conj2} is satisfied. 
\end{proof}

\begin{corollary}\label{conjGHV}
Let $Y$ be a GHV. Then
\begin{itemize}
    \item part \textup{(ii)} (the birational version) of the Cone Conjecture \ref{conj} follows from part \textup{(i)} (the automorphism version)
    \item part (2) of Conjecture \ref{conj2} is verified, \emph{i.e.} $\textup{Nef}(Y)^e=\textup{Nef}(Y)^+$. In particular $\textup{Nef}(Y)^e=\textup{Nef}(Y)^+=\textup{Nef}(Y)$
\end{itemize}
Moreover, the Morrison's Cone Conjecture \ref{conj2} is equivalent to the Kawamata's Cone Conjecture \ref{conj}.
\end{corollary}

\begin{proof}
In general the following inclusions hold: $\text{Nef}(Y)^e\subseteq\ol{\text{Mov}(Y)}^e\subseteq\text{Eff}(Y)$. Since by Proposition \ref{p1.9} we have $\text{Nef}(Y)=\text{Eff}(Y)$, it follows $\text{Nef}(Y)^e=\ol{\text{Mov}(Y)}^e=\text{Eff}(Y)$. This implies the first statement. \\
We prove that $\text{Nef}(Y)^e$ coincides $\text{Nef}(Y)^+$. 
It holds $\text{Nef}(Y)^e\subseteq \text{Nef}(Y)^+$ by \cite[Lemma 5.2]{LW23}. We observe that $\text{Nef}(Y)^+=\text{Amp}(Y)^+\subset \text{Nef}(Y)$, hence since $\text{Nef}(Y)^e=\text{Nef}(Y)$ we obtain the reverse inclusion and so the equality.
\end{proof}

\section{Proof of the main theorem} \label{sec.proof}
\noindent In this section we are going to proof Theorem \ref{main1}, namely the Cone Conjecture \ref{conj} for GHV. Due to the results of the preceding section, Proposition \ref{conjquotient} and corollary \ref{conjGHV}, it's enough to provide the existence of a rational polyhedral cone $\Pi\subset (\text{Nef}(X)^G)^+$ such that $\text{Amp}(X)^G\subset C_{\text{Aut}(X)}(G)\cdot \Pi$. To achieve this, we adopt the following strategy: we establish that the cone $\text{Amp}(X)^G$ is a homogeneous self-dual cone and that the centralizer $C_{\text{Aut}(X)}(G)$ defines an action of an arithmetic group on it. This understanding allow us to invoke the main result of reduction theory, outlined in Section \ref{sec.convexgeo}, thereby culminating in the affirmative assertion of the existence of the desired $\Pi$.\\

\noindent In the following $X$ is an abelian variety and $G\le \text{Aut}(X)$ is a finite group.
\subsection{The $G$-invariant $\bR$-algebra $\textup{End}_{\bR}(X)^G$ }
Let us recall the action defined in Theo\-rem \ref{alpha}:
\begin{center}
    \begin{tikzcd}
        \alpha: \text{End}_{\bR}(X)^\times\arrow[r] & \textup{GL}(\text{End}_{\bR}(X)) \\
        \varphi \arrow[r,mapsto] & \alpha(\varphi): l\mapsto \varphi'\circ l \circ \varphi
    \end{tikzcd}
\end{center}
which extends the action of $\text{Aut}(X)$ on $N^1(X)$ by pull back.
Let us denote by $\text{Lin}(G)$ the group generated by the linear part of every $g\in G$. Since $\text{Lin}(G)\le \text{End}
(X)^\times$, it acts on the $\bQ$-algebra $\text{End}_{\bQ}(X)$. For simplicity we say that $G$ acts on $\text{End}_{\bQ}(X)$.

\begin{definition}\label{defDG}
$\textup{End}_{\bQ}(X)^{G}:=\{ \varphi \in \textup{End}_{\bQ}(X) \mid g'\varphi g= \varphi\text{ for every } g\in \text{Lin}(G) \}$.
\end{definition}

In this subsection we prove that $\textup{End}_{\bQ}(X)^{G}$ is a finite dimensional $\bQ$-algebra with a positive definite involution. Additionally, we deduce that $\textup{End}_{\bQ}(X)^{G}\otimes \bR$ is isomorphic to a certain algebra of matrices, similar to Theorem \ref{dec-endo-algebra}. More precisely:

\begin{theorem}\label{dec-Ginvarinatalgebra}
Let $X$ be an abelian variety and $G\le \textup{Aut}(X)$ be a finite group (which does no contain any translation). 
Then: 
\begin{romanenumerate}
    \item $\textup{End}_{\bQ}(X)^G$ is a finite dimensional $\bQ$-algebra with an involution $\iota$ given by $\iota(x)=x'$ for every $x\in \textup{End}_{\bQ}(X)^G \subseteq \textup{End}_{\bQ}(X)$ which is positive-definite with respect to the trace reduce over $\bQ$.

    \item We have the following isomorphism of $\bR$ algebras: 
\begin{center}
    \begin{tikzcd}
        \Psi\colon (\textup{End}_{\bQ}(X)^{G}\otimes \bR, \iota) \arrow[r, "\simeq"] & (\prod\limits_i \textup{Mat}_{l_i}(\bR) \times \prod\limits_j \textup{Mat}_{m_i}(\bC) \times \prod\limits_k \textup{Mat}_{n_k}(\bH), \dagger)
    \end{tikzcd}
    \end{center}
where the involution $\iota$ is sent to the conjugate transpose $\dagger$ on each factor.
\end{romanenumerate}
\end{theorem}

Before proving the theorem above we recall fundamental results about finite-dimensional algebra with positive-definite involution.

\begin{definition}
An algebra $\cA$ is called \emph{simple} if $\cA^2 =\{ab \colon a,b\in \cA\} \not= {0}$ and it has no proper ideals.\\
A finite-dimensional algebra $\cA$ is said to be \emph{semisimple} if it can be expressed as a Cartesian product of simple sub-algebras.
\end{definition}

\begin{lemma}[Lemma 8.4.5 \cite{Voi21}] \label{ssalgebra}
Let $\cA_k$ be a $k$-algebra, for a subfield $k\subset \bR$, and $\tau$ be a positive-definite involution with respect to the trace. Then $\cA_k$ is semisimple.
\end{lemma}

\begin{remark}\label{invarinatdecom}
Let $\cA_k$ as in Lemma \ref{ssalgebra} and consider the decomposition $\cA_k=\prod\limits_i{\cA_i}$ into simple sub-algebra $\cA_i\subset \cA_k$. Then $\tau$ preserves this decomposition, \emph{i.e.} $\tau(\cA_i)=\cA_i$ for all $i$. Indeed if $\tau(\cA_i)=\cA_j$ for $i\not=j$, then $\cA_j$ is a simple factor and $\cA_i\cA_j=0$. Therefore $\text{Tr}(\cA_i\tau(\cA_i))=\text{Tr}(\cA_i\cA_j)=0$ which is a contradiction since $\tau$ is positive definite with respect to the trace.
\end{remark}

\begin{lemma}[Lemma 5.5.1 \cite{BL}]\label{salgebra}
For any simple $\bR$-algebra $\cA_{\bR}$ of finite dimension with a positive-definite involution there is an isomorphism of $\bR$-algebra from $(\cA_{\bR},\tau)$ to $(\textup{Mat}_{n}(\bF), \dagger)$ for some $n\in \bN$, where $\bF= \bR,\bC,\bH$ and $\dagger$ is the correspondent conjugate transpose on each field.
\end{lemma}

Let us return to our situation. Given $X$ abelian variety we have the finite dimensional $\bQ$-algebra $\text{End}_{\bQ}(X) $ with the Rosati involution $'$. The Rosati involution $'$ depends on the choice of the ample line bundle $L$. Since we are considering $X$ with an action of a finite group $G$,  we can choose $L$ to be $G$-invariant.\footnote{If $L$ is not invariant, we can consider $\sum\limits_{g\in G} g^\ast L$ which defines a $G$-invariant ample line bundle since $L\not=0$ and ample.}
Hence we have $\phi_{L}=\phi_{g^\ast L}\overset{\eqref{dual2}}{=}\hat{g}\phi_L g$ for every $g\in G$. This leads to the following relation:
\begin{equation} \label{g'isg1}
  \forall g\in G \qquad g'\overset{\eqref{rosati}}{=}\phi_L^{-1} \hat{g}\phi_L=\phi_L^{-1} \hat{g}\phi_L g g^{-1} = g^{-1}.    
\end{equation}

\begin{remark}
We remark that $\forall f,g\in \text{End}_{\bQ}(X) $ it holds 
\begin{equation} \label{rosaticontravariante}
 f'g'=(gf)', 
\end{equation}
see \cite[Section 5.1]{BL}
\end{remark}

\noindent We are in position to prove Theorem \ref{dec-Ginvarinatalgebra}.
\begin{proof}[proof of Theorem \ref{dec-Ginvarinatalgebra}]
\begin{romanenumerate}
\item 
We first prove that $\textup{End}_{\bQ}(X)^G$ is a well-defined sub-algebra of the finite dimensional $\bQ$-algebra $\textup{End}_{\bQ}(X)$, \emph{i.e.} the algebra operations of $\textup{End}_{\bQ}(X)$ are $G$-equivariant.
In the following $x,y\in \text{End}_{\bQ}(X)$, $\lambda\in \bQ$ and $g\in G$:
\begin{itemize} 
\item[1.] $\alpha(g)(x+y)=g'(x+y)g=g' x g+g'yg=\alpha(g)(x)+\alpha(g)(y)$, 
\item[2.] $\alpha(g)(\lambda x)=g'\lambda xg= \lambda \alpha(g)(x)$,
\item[3.] $\alpha(g)(xy)=g' xy g =g'xgg'yg=\alpha(g)(x)\alpha(g)(y)$ we use $g'g=id$ by \eqref{g'isg1}.
\end{itemize}
Clearly, the multiplicative and additive identity are in $\text{End}_{\bQ}(X)^G $ as well as the multiplicative and additive inverse. Thus, $\text{End}_{\bQ}(X)^G$ is a finite-dimensional $\bQ$-algebra.
We prove that the Rosati involution $'$ on $\textup{End}_{\bQ}(X) $ is $G$-equivariant, \emph{i.e.} $\iota$ defines an involution on $\text{End}_{\bQ}(X)^G$. For every $g\in G$, the following diagram is commutative:
\begin{center}
    \begin{tikzcd}
      \textup{End}_{\bQ}(X) \arrow[r," ' "] \arrow[d, "\alpha(g)" '] & \textup{End}_{\bQ}(X) \arrow[d,"\alpha(g)"]  \\
      \textup{End}_{\bQ}(X) \arrow[r, " ' " '] & \textup{End}_{\bQ}(X).
    \end{tikzcd}
\end{center}
Indeed for every $\varphi \in \text{End}_{\bQ}(X)$ it holds 
$$ \alpha(g)(\varphi')= g'\varphi'g \overset{\eqref{rosaticontravariante}}{=}(g'\varphi g)'=\bigl(\alpha(g)\varphi\bigr)'. $$ 
Therefore $\iota$ is well-defined on $\text{End}_{\bQ}(X)^G $. In particular, it is still an involution and positive-definite with respect to trace.

\item By (i) $\textup{End}_{\bQ}(X)^{G}\otimes \bR$ is a finite dimensional $\bR$-algebra with positive involution, hence by Lemma \ref{ssalgebra} it is a semisimple $\bR$-algebra, i.e,
$$\textup{End}_{\bQ}(X)^{G}\otimes \bR= \prod \limits_i \cA_i  \quad \text{with $\cA_i$'s simple $\bR$-algebras of finite dimension.}$$
Moreover, by Remark \ref{invarinatdecom} we have that $\iota$ preserves each simple factor $\cA_i$ of the above decomposition, hence $\iota_{|\cA_i}$ defines an involution on $\cA_i$ which is positive-definite for all $i$. By applying at each factor the classification of finite-dimensional simple $\bR$-algebra with a positive-definite involution, see Lemma \ref{salgebra}, we obtain the following isomorphism of $\bR$ algebras: 
\begin{center}
    \begin{tikzcd}
        \Psi\colon (\textup{End}_{\bQ}(X)^{G}\otimes \bR, \iota) \arrow[r, "\simeq"] & (\prod\limits_i \text{Mat}_{l_i}(\bR) \times \prod\limits_j \text{Mat}_{m_i}(\bC) \times \prod\limits_k \text{Mat}_{n_k}(\bH), \dagger)
    \end{tikzcd}
    \end{center}
\end{romanenumerate}
\end{proof}

\begin{remark}
Using the definition of the action of $G$ on $\text{End}_{\bQ}(X)$ (Theorem \ref{alpha}), we have that $$\text{End}_{\bR}(X)^G=\textup{End}_{\bQ}(X)^{G}\otimes \bR. $$
\end{remark}

\subsection{The $G$-invariant ample cone of $X$}
In this sub-section we prove that the $G$-invariant ample cone $\text{Amp}(X)^G$ is a homogeneous self-dual cone.

\begin{definition}
We define the $\bR$-vector space $N^1(X)^G:=\{ D\in N^1(X) \colon \text{ $G$-invariant}\}$.
We define the $G$-invariant ample cone $\text{Amp}(X)^G:=\{ D\in \text{Amp}(X) \colon \text{ $G$-invariant}\}$.
\end{definition}

\begin{theorem}\label{Ginvariantcone}
Let $X$ be abelian variety and $G\le \textup{Aut}(X)$ be a finite group. Then the $G$-invariant ample cone is isomorphic to :
$$ \textup{Amp}(X)^G \xrightarrow{\quad \simeq \quad} \bigoplus\limits_i \cP_{l_i}(\bR)  \oplus  \bigoplus\limits_j \cP_{m_i}(\bC)  \oplus  \bigoplus\limits_k \cP_{n_k}(\bH)\subseteq \Psi(\textup{End}_{\bR}(X)^G)$$
where $\cP_{l}(\bF)$ is the cone of positive-definite hermitian matrices of dimension $l$ over the field $\bF$ and $\Psi$ is defined in Theorem \ref{dec-Ginvarinatalgebra}.
In particular, it is a homogeneous self-dual cone.
\end{theorem}
\begin{proof}

By Theorem \ref{thm2.3}, we have the following isomorphism of $\bR$-vector spaces:

\begin{center} \begin{tikzcd}
f \colon N^1(X) \arrow[r, "\simeq"] & \textup{End}^s_{\bR}(X) \\
 D \arrow[r, mapsto] & \phi_L^{-1} \phi_D,
\end{tikzcd}
\end{center}
where $\textup{End}^s_{\bR}(X)$ denoted the space of $\bR$-endomorphisms on $X$ fixed by the Rosati involution.
In Theorem \ref{alpha}, we prove that the action $\alpha\colon \textup{End}_{\bR}(X)^\times\longrightarrow \textup{End}_{\bR}(X)$ extends the action of $\textup{Aut}(X)$ on $N^1(X)$ by pullbacks, \emph{i.e.} for every $g \in G$ we have the following commutative diagram:
\begin{center}
    
\begin{tikzcd}
N^{1}(X) \arrow[rr, "f"] \arrow[d, "g^{\ast}" '] &  & \textup{End}^s_{\bR}(X) \arrow[d, "\alpha(g)"] \\
N^{1}(X) \arrow[rr, "f" ']  &  & \textup{End}^s_{\bR}(X)                       
\end{tikzcd}
\end{center}

Therefore:
\begin{equation}\label{ginvarinatNS}
\begin{split}
f(N^1(X)^{G}) &= (\textup{End}^s_{\bR}(X))^{G}\\
&=\{ x\in \text{End}_{\bR}(X) \mid  x'=x \text{ and $\alpha(g)(x)=x$ for every $g\in G$}\}\\
&=\{ x\in \text{End}_{\bR} (X)^{G} \mid \iota(x)=x\} .
\end{split}   
\end{equation}

By Theorem \ref{dec-Ginvarinatalgebra} we have the following isomorphism of $\bR$-algebras 
\begin{center}
    \begin{tikzcd}
        \Psi\colon (\text{End}_{\bR}(X)^{G}, \iota) \arrow[r, "\simeq"] & (\prod\limits_i \text{Mat}_{l_i}(\bR) \times \prod\limits_j \text{Mat}_{m_i}(\bC) \times \prod\limits_k \text{Mat}_{n_k}(\bH), \dagger)
    \end{tikzcd}
\end{center}
which combines with equation \eqref{ginvarinatNS} yields to the following isomorphism of vector spaces
\begin{center}
    \begin{tikzcd}
      (\Psi\circ f) \colon N^1(X)^G \arrow[r, "\simeq " ] & \bigoplus\limits_i \cH_{l_i}(\bR) \oplus \bigoplus\limits_j \cH_{m_i}(\bC)  \oplus  \bigoplus\limits_k \cH_{n_k}(\bH) \subset \Psi(\textup{End}_{\bR}(X)^G)
    \end{tikzcd}
\end{center}
where we use additive notation for $N^1(X)^G$ to emphasise that it need not be a sub-algebra of $\text{End}_{\bR}(X)^G$.
By \cite[Remark 5.2.5]{BL} we know that the embedding $f$  establishes a bijection between ample line bundles on $X$ and totally positive endomorphisms in $\text{End}^s_{\bQ}(X)$, \emph{i.e.} endomorphisms in $\text{End}^s_{\bQ}(X)$ such that all the zeros of their characteristic polynomial are positive. Therefore, we obtain:
\begin{center}
    \begin{tikzcd}
      (\Psi\circ f) \colon \text{Amp}(X)^G \arrow[r, "\simeq " ] & \bigoplus\limits_i \cP_{l_i}(\bR) \oplus \bigoplus\limits_j \cP_{m_i}(\bC) \oplus \bigoplus\limits_k \cP_{n_k}(\bH) \subset \Psi(\textup{End}_{\bR}(X)^G).  
    \end{tikzcd}
\end{center}

By Theorem \ref{thm1.2} we know that a cone $C=\bigoplus\limits_{i}C_i$ is homogeneous self-dual if and only if each indecomposable cone $C_i$'s is too. 
By the classification of the homogeneous self-dual indecomposable cones, see Theorem \ref{thm1.2}, we know that each factor $\cP_{l}(\bF)$ of $\text{Amp}(X)^G$ is a homogeneous self-dual cone, hence we have that $\text{Amp}(X)^G$ is a homogeneous self-dual cone.

\end{proof}

\subsection{Action of the centralizer}
We prove that the centralizer $\textup{C}_{\textup{Aut}(X)}(G)$ defines an arithmetic subgroup in $\text{Aut}(\text{Amp}(X)^G)^0$.

\begin{lemma}\label{end-arithm}
Let $X$ be an abelian variety and $G\le \textup{Aut}(X)$ be finite group. The group of units $\bigl(\textup{End}_{\bR}(X)^G\bigl)^\times$ is an affine algebraic group defined over $\bQ$ and $\textup{C}_{\textup{End}(X)^\times}(G)$ is an arithmetic subgroup.
\end{lemma}
\begin{proof}
By Theorem \ref{dec-Ginvarinatalgebra} part (i) $\text{End}_\bQ(X)^G$, as finite dimensional $\bQ$-algebra, is a finite dimensional $\bQ$-vector space and, since $\text{End}_\bR(X)^G=\text{End}_\bQ(X)^G\otimes\bR$, it defines a $\bQ$-structure on  $\text{End}_\bR(X)^G$. We set the following isomorphism as affine spaces:
 $$\text{End}_{\bQ}(X)^G \simeq \bQ^d \qquad \text{End}_{\bR}(X)^G \simeq \bQ^d\otimes \bR\simeq \bR^d. $$ 
We denote by $\cA_{\bQ}$ the $d$-dimensional $\bQ$-algebra $\text{End}_{\bQ}(X)^G$ and by $\cA_{\bR}$ its extension over $\bR$, so $\text{End}_{\bR}(X)^G$.
We consider the following injective map:
\begin{center}
    \begin{tikzcd}
     j\colon (\cA_{\bR})^\times \arrow[r,hookrightarrow] & \cA_{\bR}\times \cA_{\bR} \\
     x \arrow[r,mapsto] & (x,x^{-1}) .
    \end{tikzcd}
\end{center}
This map yields to the following description of $(\cA_{\bR})^\times$ as Zariski closed in the affine space $\cA_{\bR}\times \cA_{\bR}\simeq\bR^{2d}$:
\begin{center}
$(\cA_{\bR})^\times\overset{j}{\simeq} \{(x,y)\in \cA_{\bR}\times \cA_{\bR} \mid xy-1=0\}=V(xy-1) \subset \cA_{\bR}\times \cA_{\bR}\simeq\bR^{2d}. $ 
\end{center}
Therefore, $(\cA_{\bR})^\times$ is an affine algebraic subgroup $\cG$ of the affine space $\bR^{2d}$. Since the equation defining $(\cA_{\bR})^\times$ is in fact over $\bQ$, we have that $(\cA_{\bR})^\times$ is an affine variety defined over $\bQ$. Moreover, we observe that the group of $\bQ$-points of $(\cA_{\bR})^\times$ is $(\cA_{\bQ})^\times$.\\
We now prove the arithmetic part. First we recall that by \eqref{g'isg1} we have $g'=g^{-1}$ for all $g\in G$, hence:
$$\textup{C}_{\textup{End}(X)^\times }(G)= 
\{ \varphi \in \textup{End}(X)^\times \mid g'\varphi g= \varphi\text{ for every } g\in \text{Lin}(G) \}=\bigl(\text{End}(X)^G\bigr)^\times. $$
Moreover, given an abelian variety $X \simeq \bC^{n}/ \Lambda$ there is the following faithful representation (called \emph{rational representation}): 
\begin{center}
    \begin{tikzcd}
        \rho \colon \text{End}(X) \arrow[r,hookrightarrow] & \text{End}_{\bZ}(\Lambda)\simeq \text{Mat}_{2n}(\bZ) \\
        \varphi \arrow[r,mapsto] & \widetilde{\varphi}
    \end{tikzcd}
\end{center}
where $ \widetilde{\varphi}$ is the unique $\bC$-linear map such that $\widetilde{\varphi}(\Lambda)\subseteq \Lambda$ inducing $\varphi$.
We can extend $\rho$ $\bR$-linearly obtaining the following groups monomorphism
$$\rho_{\bR} \colon \text{End}_{\bR}(X) \xhookrightarrow{\qquad \qquad} \text{Mat}_{2n}(\bR). $$
Moreover, we restrict $\rho_{\bR}$ to the group of units of $G$-invariant $\bR$-endomorphisms, yielding the following embedding:
\begin{center}
    \begin{tikzcd}
        \rho_{\bR} \colon \bigl(\text{End}_{\bR}(X)^G\bigr)^\times \arrow[r,hookrightarrow] & \text{GL}_{\bR}(\Lambda\otimes \bR)\simeq \text{GL}_{2n}(\bR)
    \end{tikzcd}
\end{center}
which, in particular, is a morphism of $\bQ$-algebraic groups, \emph{i.e.} it is a closed embedding. 
Since $
\rho_{\bR}\biggl(\bigl(\text{End}(X)^G\bigr)^\times\biggr) \subset \text{GL}_{2n}(\bZ)$, we obtain:
$$\bigl(\text{End}(X)^G\bigr)^\times\simeq \text{Im}(\rho_{\bR})\cap \text{GL}_{2n}(\bZ)=\rho_r\biggl(\bigl(\text{End}_{\bR}(X)^G\bigr)^\times\biggr)\cap \text{GL}_{2n}(\bZ).$$
Therefore, denoting the algebraic group $\cG= \rho_{\bR}\biggl(\bigl(\text{End}_{\bR}(X)^G\bigr)^\times\biggr)$ we have$$\cG(\bZ)\overset{\rho_{\bR}}{\simeq}\bigl(\text{End}(X)^G\bigr)^\times, $$
which by Definition \ref{defari} proves that $\textup{C}_{\textup{End}(X)^\times }(G)=\bigl(\text{End}(X)^G\bigr)^\times$ is an arithmetic group in $\bigl(\text{End}_{\bR}(X)^G\bigr)^\times$.

\end{proof}

The following lemma tell us that when we look at the pull-back action of $\text{Aut}(X)$ on $N^1(X)$ we can forget about the translation.

\begin{lemma}\label{end=aut}
Let $X$ be an abelian variety and let us consider the following homomorphism: 
\begin{center}
\begin{tikzcd} \alpha \colon \textup{Aut}(X) \arrow[r] & \textup{GL}(N^{1}(X)) \\
    \varphi  \arrow[r,mapsto] & (\varphi^{*} \colon D \xmapsto{\qquad} \varphi^{*}D).
\end{tikzcd}
\end{center}
Then, $\alpha(\textup{Aut}(X))=\alpha(\textup{End}(X)^{\times})$. In other words, every translation of $X$ acts as the identity on $N^1(X)$.
\end{lemma}
\begin{proof}
We prove that for every translation $t\in \text{Aut}(X)$ on $X$ then $\alpha(t)=t^\ast=id_{N^1(X)}$. 
Let us denote by $[D]_{num}\in N^1(X)=(\text{Div}(X)/\equiv)\otimes \bR$ the numerical class of $D$. We recall that on $N^1(X)$ the equivalence coincides with the algebraic equivalence, see \cite[Remark 1.1.21]{Laz03}, therefore $[D]_{num}=[D]_{alg}$ in $N^1(X)$.
Let $D$ be a Cartier divisor, we have $\cO_X(D-t^\ast D) \in \textup{Pic}^{0}(X)$. Therefore, since $\text{Div}(X)/\sim_{alg}\simeq \text{Pic}(X)/\text{Pic}^0(X)$ we obtain $[D-t^\ast D]_{alg}=[D-t^\ast D]_{num}=0$. Thus $t^\ast=id_{N^1(X)}$
and $\alpha(\textup{Aut}(X))=\alpha(\textup{End}(X)^{\times})$.
\end{proof}

\begin{lemma} \label{alphaG}
Let $X$ be an abelian variety and $G\le \textup{Aut}(X)$ finite group.
The following $\bQ$-morphism of algebraic groups
\begin{center}\begin{tikzcd} 
\rho \colon \bigl(\textup{End}_{\bR}(X)^{G}\bigr)^{\times} \arrow[r] & \textup{Aut}(\textup{Amp}(X)^{G})^{0} \subset \textup{GL}(N^1(X)^G)\\  \varphi \arrow[r,mapsto ] & \rho(\varphi)\colon x \mapsto \iota(\varphi)\circ x \circ \varphi \end{tikzcd} \end{center}
is surjective.
\end{lemma}

\begin{proof}
It is clear that is a morphism of algebraic groups.
Moreover it is also well-defined as morphism of $\bQ$-varieties $\bigl(\text{End}_{\bQ}(X)^G\bigr)^\times\longrightarrow \text{GL}(N^1_{\bQ}(X)^G)$.\\
Assume, for simplicity, that $\text{End}_{\bR}(X)^{G}$ has a single direct factor. Using the notation of Theorem \ref{dec-Ginvarinatalgebra}: 
\begin{align*}
 &\Psi\colon \bigl(\text{End}_{\bR}(X)^{G}, \iota\bigr) \xrightarrow{\quad\simeq\quad} \bigl(\text{Mat}_{l}(\bF),\dagger \bigr)\\
 & \Psi\colon \bigl(\text{End}_{\bR}(X)^{G}\bigr)^{\times}\xrightarrow{\quad\simeq\quad} \text{GL}_l(\bF)
\end{align*}
and by Theorem \ref{Ginvariantcone} we have $${(\Psi\circ f)}\colon \text{Amp}(X)^G\xrightarrow{\quad\simeq\quad} \cP_l(\bF)$$
where $\bF = \bR,\bC,\bH$. Thus we have:
\begin{center}
\begin{tikzcd} 
\bigl(\text{End}_{\bR}(X)^{G}\bigr)^{\times}\arrow[d,"\rsimeq" ', "\Psi"]  \arrow[r,"\rho"] & \textup{Aut}(\text{Amp}(X)^{G})^{0} \arrow[d,"\lsimeq", "\Psi" ']  & \varphi \arrow[d,mapsto, "\Psi"'] \arrow[r,mapsto] & \rho(\varphi)\colon x \mapsto \iota(\varphi) \circ x \circ \varphi \arrow[d,mapsto, "\Psi"] 
\\
 GL_l(\bF) \arrow[r,"\rho'"] & \textup{Aut}(\cP_l(\bF))^{0} & M \arrow[r,mapsto] & \rho(M)\colon D \mapsto M^\dag D M 
\end{tikzcd} 
\end{center}
Theorem \ref{groupofcone} guarantees the surjectivity of $\rho$ since every automorphism in $\textup{Aut}(\cP_l(\bF))^{0}$ are of the form $D \mapsto M^\dag D M  $ with $M\in GL_l(\bF)$. The proof can be generalized in the case $\text{End}_{\bR}(X)^{G} \simeq  \prod\limits_i \text{Mat}_{l_i}(\bR) \times \prod\limits_j \text{Mat}_{m_i}(\bC) \times \prod\limits_k \text{Mat}_{n_k}(\bH)$, since for a $C=\bigoplus\limits_i C_i$ the identity component $\textup{Aut}(C)^{0}$ is isomorphic to $\prod\limits_i\textup{Aut}(C_i)^{0}$ by Theorem \ref{thm1.2}. Thus, the surjectivity of $\rho$ follows by applying the previous proof at each factor. 
\end{proof}

\begin{remark}\label{rho_alpha}
We note that by definition $\rho$ of Lemma \ref{alphaG} is nothing else that the restriction of \begin{center}
    \begin{tikzcd}
        \alpha: \text{End}_{\bR}(X)^\times\arrow[r] & \textup{GL}(\text{End}_{\bR}(X)) \\
        \varphi \arrow[r,mapsto] & \alpha(\varphi): l\mapsto \varphi'\circ l \circ \varphi
    \end{tikzcd}
\end{center}
to the $G$-invariant endomorphism. In particular since by Theorem \ref{alpha} $\alpha$ is the extension of the action of $\text{Aut}(X)$ by pull-backs on $N^1(X)$, we deduce that $\rho$ is the extension of the action of $\text{C}_{\text{Aut}(X)}(G)$ by pull-backs on $N^1(X)^G$. Therefore, by the previous Lemma we deduce:
$$ \rho(\textup{C}_{\textup{Aut}(X)}(G))=\rho(\textup{C}_{\textup{End}^\times(X)}(G)). $$
\end{remark}

\begin{proposition}\label{cent_arith}
    Let $X$ be an abelian variety and $G\le \textup{Aut}(X)$ be a finite group. Then the centralizer $\textup{C}_{\textup{Aut}(X)}(G)$ defines an arithmetic subgroup in $\textup{Aut}(\textup{Amp}(X)^G)^0$.
\end{proposition}
\begin{proof}
Let us consider the $\bQ$-morphism of algebraic groups:
      $$\rho \colon \bigl(\textup{End}_{\bR}(X)^{G}\bigr)^{\times} \longrightarrow  \text{GL}\bigl(N^1(X)^G\bigr). $$
     Since, by Lemma \ref{alphaG}, $\textup{Aut}(\textup{Amp}(X)^G)^0$ is the image of $\rho$ we obtain that it is an algebraic group defined over $\bQ$. By Remark \ref{rho_alpha} we have that $\rho(\textup{C}_{\textup{Aut}(X)}(G))=\rho(\textup{C}_{\textup{End}^\times(X)}(G))$. Moreover, by Lemma \ref{end-arithm} $\textup{C}_{\textup{End}^\times(X)}(G)$ is an arithmetic subgroup of $\bigl(\textup{End}_{\bR}(X)^{G}\bigr)^{\times}$
     and since the property to be arithmetic is preserved under $\bQ$-epimorphism, see \cite[Remark 8.22]{Bor19}, we obtain that $\rho(\textup{C}_{\textup{Aut}(X)}(G))$ is an arithmetic subgroup in $\textup{Aut}(\textup{Amp}(X)^G)^0$.
\end{proof}

\subsection{Proof of Theorem \ref{main1}}
We are in position to prove Theorem \ref{main1}.

\begin{proof}[Proof of Theorem \ref{main1}]
In view of Proposition \ref{conjquotient} and Corollary \ref{conjGHV}, it is sufficient to provide a rational polyhedral cone $\Pi'\subset \bigl(\text{Nef}(X)^G\bigr)^+$ such that $\text{Amp}(X)^G\subset H\cdot \Pi'$ for some $H\le \text{N}_{\text{Aut}(X)}(G)$.
By Theorem \ref{Ginvariantcone}, $\text{Amp}(X)^G$ is a homogeneous self-dual cone.
By Proposition \ref{cent_arith}, $\rho(\textup{C}_{\textup{Aut}(X)}(G))$ is an arithmetic subgroup of $\text{Aut}(\text{Amp}(X)^{G})^{0}$. By applying Theorem \ref{thm1.1} of reductive theory: there exists a rational polyhedral cone $\Pi' \subset \bigl(\text{Nef}(X)^G\bigr)^+$ such that $\text{Amp}(X)^G\subset (\textup{C}_{\textup{Aut}(X)}(G)\cdot \Pi') $. 
\end{proof}

\bibliographystyle{plain}
\bibliography{bibliography}
 
\end{document}